\documentclass[12pt]{article}
\usepackage[a4paper,bindingoffset=0.2in,%
            left=1in,right=1in,top=1in,bottom=1in,%
            footskip=.25in,
            ]{geometry}  
\usepackage{xcolor}
\usepackage[linesnumbered,ruled,vlined]{algorithm2e}
\usepackage{amsthm,amsmath,amssymb}
\usepackage{tikz}
\usetikzlibrary{arrows}
\usepackage{mathrsfs}
\usepackage{graphicx}
\usepackage{array}
\usepackage{enumerate}
\usepackage[english]{babel}
\usepackage{longtable}
\usepackage{mathrsfs}
\usepackage{authblk}
\usepackage{placeins}
\usepackage[backend=biber, style=numeric, sorting=none, sortcites, maxcitenames=50]{biblatex}
\usepackage[colorlinks=true,citecolor=blue]{hyperref}

\newtheorem{theorem}{Theorem}

\newtheorem{construction}{Construction}

\AtEveryBibitem{\clearlist{language}} 

\title{Graceful coloring is computationally hard} 
\author[1]{Cyriac Antony}
\author[2]{Laavanya D.}
\author[3]{Devi Yamini S.}
\affil[1]{IIT Madras, Chennai, India}
\affil[2,3]{Vellore Institute of Technology, Chennai, India}

\date{ }

\begin{document}
\maketitle
\begin{abstract}
Given a (proper) vertex coloring \( f \) of a graph \( G \), say \( f\colon V(G)\to \mathbb{N} \), the \emph{difference edge labelling} induced by \( f \) is a function \( h\colon E(G)\to \mathbb{N} \) defined as \( h(uv)=|f(u)-f(v)| \) for every edge \( uv \) of \( G \). 
A \emph{graceful coloring} of \( G \) is a vertex coloring \( f \) of \( G \) such that the difference edge labelling \( h\) induced by \( f \) is a (proper) edge coloring of \( G \). 
A graceful coloring with range \( \{1,2,\dots,k\} \) is called a graceful \( k \)-coloring. 
The least integer \( k \) such that \( G \) admits a graceful \( k \)-coloring is called the \emph{graceful chromatic number} of \( G \), denoted by \( \chi_g(G) \).

We prove that \( \chi(G^2)\leq \chi_g(G)\leq a(\chi(G^2)) \) for every graph \( G \), where \( a(n) \) denotes the \( n \)th term of the integer sequence A065825 in OEIS. 
We also prove that graceful coloring problem is NP-hard for planar bipartite graphs, regular graphs and 2-degenerate graphs. 
In particular, we show that for each \( k\geq 5 \), it is NP-complete to check whether a planar bipartite graph of maximum degree \( k-2 \) is graceful \( k \)-colorable. 
The complexity of checking whether a planar graph is graceful 4-colorable remains open. 
\end{abstract}


\section{Introduction}
Many branches of mathematics started out as problems in recreational mathematics which are easy to understand, yet challenging to solve. 
The story of graph theory is no different. 
The innocuous problem of coloring maps using only 4 colors gave birth to a thriving area of graph theory named graph coloring. 
The notion of graph labelling is a generalisation of graph coloring. 
Graph labelling is an area of immense theoretical interest and diverse practical applications, evident from Gallian's dynamic survey~\cite{gallian}. 
Similar to how attempts to prove the four color conjecture lead to the historical origin or popularity of the area of graph colorings, study of graceful labelling and harmonious labelling lead to the boom of the area of graph labellings~\cite{gallian}. 

The definition of graceful labelling requires the notion of difference edge labelling induced by a vertex labelling. 
Given a vertex labelling \( f \) of a graph \( G \), say \mbox{\( f\colon V(G)\to \mathbb{N} \)}, the \emph{difference edge labelling induced by \( f \)} is a function \( h\colon E(G)\to \mathbb{N}\cup \{0\} \) defined as \( h(uv)=|f(u)-f(v)| \) for every edge \( uv \) of \( G \). 
A \emph{graceful labelling} of a graph \( G \) on \( m \) edges is an injection \( f\colon V(G)\to \{0,1,\dots,m\} \) such that the difference edge labelling \( h \) induced by \( f \) is an injection from \( E(G) \) to \( \{1,2,\dots,m\} \)~\cite{gallian}. 

The most popular problem on graceful labelling is settling the infamous Kötzig-Ringel-Rosa conjecture, better known as the graceful tree conjecture, which states that all trees are graceful. 
The graceful tree conjecture is far from resolved till date. 
Hence, the practical approaches to the problem includes resolving the conjecture for subclasses of trees on one hand, and resolving weaker versions of the conjecture on the other hand. 
One way to produce notions weaker than graceful labelling is to impose restrictions locally on vertex neighbourhoods rather than globally. 
This produces the notion of graceful coloring. 
A vertex labelling \( f \) of a graph \( G \), say \( f\colon V(G)\to \mathbb{N} \), is a \emph{graceful coloring} of \( G \) if (i)~\( f \) is an injection when limited to each vertex neighbourhood, and (ii)~the induced difference edge labelling \( h \) is an injection when limited to each vertex neighbourhood; formally, the restriction \( f_{\restriction N[v]} \) is an injection and \( h_{\restriction G[N[v]]} \) is an injection for the closed neighbourhood \( N[v] \) of each vertex \( v \) in \( G \). 
In other words, a graceful coloring of \( G \) is a (proper vertex) coloring \( f \) of \( G \) such that the difference edge labelling \( h\) induced by \( f \) is a (proper) edge coloring of \( G \)~\cite{bi_byers} (see Figure~\ref{fig:eg graceful coloring} for an example). 
The graceful coloring first appeared in Bi et al.~\cite{bi_byers}, and was studied in more detail in Byers' thesis~\cite{byers}. 

\begin{figure}[hbtp] 
	\centering
\tikzset{every picture/.style={line width=0.75pt}} 
\begin{tikzpicture}[x=0.75pt,y=0.75pt,yscale=-1,xscale=1]

\draw   (40,25) .. controls (40,19.48) and (44.48,15) .. (50,15) .. controls (55.52,15) and (60,19.48) .. (60,25) .. controls (60,30.52) and (55.52,35) .. (50,35) .. controls (44.48,35) and (40,30.52) .. (40,25) -- cycle ;
\draw   (139,28) .. controls (139,22.48) and (143.48,18) .. (149,18) .. controls (154.52,18) and (159,22.48) .. (159,28) .. controls (159,33.52) and (154.52,38) .. (149,38) .. controls (143.48,38) and (139,33.52) .. (139,28) -- cycle ;
\draw   (140,128) .. controls (140,122.48) and (144.48,118) .. (150,118) .. controls (155.52,118) and (160,122.48) .. (160,128) .. controls (160,133.52) and (155.52,138) .. (150,138) .. controls (144.48,138) and (140,133.52) .. (140,128) -- cycle ;
\draw   (40,127) .. controls (40,121.48) and (44.48,117) .. (50,117) .. controls (55.52,117) and (60,121.48) .. (60,127) .. controls (60,132.52) and (55.52,137) .. (50,137) .. controls (44.48,137) and (40,132.52) .. (40,127) -- cycle ;
\draw    (56,32) -- (142.5,121) ;
\draw    (60,127) -- (140,128) ;
\draw    (60,26) -- (140,27) ;
\draw    (149,38) -- (150,118) ;
\draw    (50,36) -- (51,116) ;
\draw    (160,127) -- (240,128) ;
\draw   (240,127) .. controls (240,121.48) and (244.48,117) .. (250,117) .. controls (255.52,117) and (260,121.48) .. (260,127) .. controls (260,132.52) and (255.52,137) .. (250,137) .. controls (244.48,137) and (240,132.52) .. (240,127) -- cycle ;
\draw    (159,28) -- (245.5,117) ;

\draw (145.5,123) node [anchor=north west][inner sep=0.75pt]  [font=\footnotesize] [align=left] {1};
\draw (245.5,121) node [anchor=north west][inner sep=0.75pt]  [font=\footnotesize] [align=left] {3};
\draw (45,19) node [anchor=north west][inner sep=0.75pt]  [font=\footnotesize] [align=left] {2};
\draw (145.5,22) node [anchor=north west][inner sep=0.75pt]  [font=\footnotesize] [align=left] {4};
\draw (153,65) node [anchor=north west][inner sep=0.75pt]  [font=\footnotesize] [align=left] {3};
\draw (100,114) node [anchor=north west][inner sep=0.75pt]  [font=\footnotesize] [align=left] {4};
\draw (99,12) node [anchor=north west][inner sep=0.75pt]  [font=\footnotesize] [align=left] {2};
\draw (205,57) node [anchor=north west][inner sep=0.75pt]  [font=\footnotesize] [align=left] {1};
\draw (102,65) node [anchor=north west][inner sep=0.75pt]  [font=\footnotesize] [align=left] {1};
\draw (46,122) node [anchor=north west][inner sep=0.75pt]  [font=\footnotesize] [align=left] {5};
\draw (38,64) node [anchor=north west][inner sep=0.75pt]  [font=\footnotesize] [align=left] {3};
\draw (193,113) node [anchor=north west][inner sep=0.75pt]  [font=\footnotesize] [align=left] {2};

\end{tikzpicture}
\caption{A graceful coloring of a graph. This is a graceful 5-coloring.}
\label{fig:eg graceful coloring}
\end{figure}
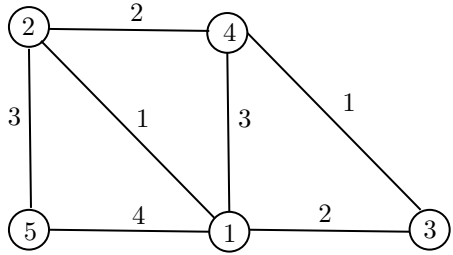

For \( k\in \mathbb{N} \), a \emph{graceful \( k \)-coloring} of \( G \) is a graceful coloring of \( G \) with range \( \{1,2,\dots,k\} \) (i.e., \( f\colon V(G)\to \{1,2,\dots,k \} \)). 
The least integer \( k \) such that \( G \) admits a graceful \( k \)-coloring is called the \emph{graceful chromatic number} of \( G \), denoted by \( \chi_g(G) \). 

It is easy to observe that under a graceful coloring, no two neighbours of a vertex can get the same color. 
A (proper) coloring of a graph \( G \) with this property is called a distance-two coloring. 
Hence, every graceful coloring is a distance-two coloring, but the converse is not true. 
The least number of colors required to produce a distance-two coloring of a graph \( G \) is called the distance-two chromatic number of \( G \), and is equivalent to the chromatic number of the square graph \( G^2 \). 
We denote the distance-two chromatic number of \( G \) by \( \chi(G^2) \). 
Obviously, \( \chi(G^2)\leq \chi_g(G) \).

We relate the graceful chromatic number of complete graphs to integer sequences. 
Throughout this paper, \( a(n) \) denotes the \( n \)th term of the integer sequence A065825 in OEIS~\cite{a065825-oeis}. 
It is known that \( \chi_g(K_n)=a(n) \)~\cite{laavanya_deviYamini}. 

In this paper, we prove that \( \chi(G^2)\leq \chi_g(G)\leq a(\chi(G^2)) \) for every graph \( G \). 
In addition, we prove that graceful coloring problem is NP-hard for planar bipartite graphs, regular graphs and 2-degenerate graphs. 
We show that (i)~for each \( k\geq 6 \), it is NP-complete to check whether a planar bipartite 3-degenerate graph of maximum degree \( k-2 \) is graceful \( k \)-colorable, (ii)~it is NP-complete to check whether a 3-regular 3\nobreakdash-connected planar bipartite graph is graceful 5-colorable, and (iii)~it is NP-complete to check whether a 2-degenerate graph of maximum degree 3 is graceful 4\nobreakdash-colorable. 
The complexity of checking whether a planar graph is graceful 4-colorable remains open. 

For brevity, we present overviews in the paper, and relegate details to the extended version of the paper. 

\section{Results}

\subsection{Graceful Coloring and Integer Sequences}
By definition, \( a(n) \) is the least integer \( k \) for which \( \{1,2,\dots,k\} \) contains a subset \( S \) of cardinality \( n \) such that no three distinct elements \( i,j,k \) of \( S \) satisfy \( |i-j|=|j-k| \). 
This imply the following. 
\begin{theorem}[\cite{laavanya_deviYamini}]
\( \chi_g(K_n)=a(n) \) for every positive integer \( n \).
\qed
\end{theorem}

If \( f \) is a distance-two \( q \)-coloring of a graph \( G \) and \( h \) is a graceful coloring of the complete graph \( K_q \) with vertex set \( \{1,2,\dots,q\} \), then \( h \circ f \) is a graceful coloring of \( G \). 
Thus, we have the following theorem. 
\begin{theorem}
\( \chi(G^2)\leq \chi_g(G)\leq a(\chi(G^2)) \) for every graph \( G \). 
\qed
\end{theorem}

\subsection{Complexity of Graceful Coloring}
The problem \textsc{Graceful \( k \)-Colorability} is defined as follows. 
\begin{center}
\fbox {\begin{minipage}{0.6\textwidth}
\noindent \textsc{Graceful \( k \)-Colorability}\\
Instance: A graph \( G \).\\
Question: Does \( G \) admit a graceful \( k \)-coloring?
\end{minipage}}\\
\end{center}

First, we show that \textsc{Graceful \( k \)-colorablity} of planar graphs is NP-complete for all \( k\geq 5 \). 
As a prelude to the reduction, we point out the following. 
\begin{theorem}
A 3-reular graph \( G \) is graceful 5-colorable if and only if \( G \) is distance-two 4-colorable.
\qed
\end{theorem}

\begin{construction}\label{constr:1}
~\\
\emph{Parameter:} An integer \( k\geq 5 \) (not part of input). \\
\emph{Input:} A 3-regular graph \( G \).\\
\emph{Output:} A graph \( G' \) of maximum degree \( k-2 \).\\
\emph{Guarantee:} \( G' \) is graceful \( k \)-colorable if and only if \( G \) is distance-two 4-colorable.\\
\emph{Steps:} Introduce a copy of \( G \). 
Attach \( k-5 \) leaf vertices at each vertex of the copy of~\( G \). 
\end{construction}
\begin{proof}[Proof of Guarantee (overview)]
Suppose that \( G' \) is graceful \( k \)-colorable. 
That is, \( G' \) admits a graceful \( k \)-coloring \( f'\colon V(G')\to \{1,2,\dots,k\} \). 
If \( v \) is a vertex of \( G' \) colored 3 by \( f' \), then \( v \) has at most \( k-3 \) neighbours in \( G' \). 
As a result, no non-leaf vertex \( v \) of \( G' \) can get color~3 (because \( d_{G'}(v)=k-2 \)). 
Similarly, each non-leaf vertex \( v \) of \( G' \) cannot get any of the colors \( 4,5,\dots,k-2 \), and thus \( v \) can get only the colors \( 1,2,k-1 \) or \( k \) under \( f \).  
Hence, the restriction of \( f \) to \( V(G) \) is a graceful coloring and in particular a distance-two coloring of \( G \) that uses only 4 colors (namely, \( 1,2,k-1 \) and \( k \)). 
Therefore, \( G \) is distance-two 4-colorable. 

Conversely, suppose that \( G \) is distance-two 4-colorable. 
Then, there exists a distance-two 4-coloring \( f \) of \( G \) with color paletter \( \{1,2,k-1,k\} \) (i.e., \( f\colon V(G)\to \{1,2,k-1,k\} \)). 
We show that \( f \) can be extended into a graceful \( k \)-coloring \( f' \) of \( G' \). 
For each non-leaf vertex \( v \) of \( G' \), define \( f'(v)=f(v) \), and color the leaf neighbours of \( v \) in \( G' \) as follows: if \( f'(v)=2 \), then color the leaf neighbours of \( v \) with colors \( \{4,5,\dots,k-2\} \) in a bijective fashion; if \( f'(v)=k-1 \), then color the leaf neighbours of \( v \) with colors \( \{3,4,\dots,k-3\} \) in a bijective fashion; if \( f'(v)\in\{1,k\} \), then color the leaf neighbours of \( v \) with distinct colors from \( \{3,4,\dots,k-2\} \). 
We complete the proof by showing that \( f' \) is a graceful \( k \)-coloring of \( G' \).
%
\end{proof}
Observe that for \( k=5 \), the output graph in Construction~\ref{constr:1} is the same as the input graph (i.e., \( G'=G \)). 
Moreover, the construction obviously takes only time polynomial in the input size. 
Also, observe that Construction~\ref{constr:1} preserves planarity and bipartiteness. 
Further, Feder, Hell and Subi~\cite{feder} proved that it is NP-complete to check whether a 3-regular 3-connected planar bipartite graph is distance-two 4\nobreakdash-colorable. 
Thanks to Construction~\ref{constr:1}, this imply the following. 
\begin{theorem}
\textsc{Graceful 5-colorablity} is NP-complete for 3-regular 3-connected planar bipartite graphs. 
\qed 
\end{theorem}
\begin{theorem}
For \( k\geq 6 \),  \textsc{Graceful \( k \)-colorablity} is NP-complete for planar bipartite graphs of maximum degree \( k-2 \). 
\qed 
\end{theorem}

Next, we show that \textsc{Graceful 4-colorablity} is NP-complete by reducing from the following problem. 

\begin{center}
\fbox {\begin{minipage}{33em}
\noindent \textsc{Positive Not-All-Equal $3$-Sat E4}\\
\emph{Instance:} A set \( X \) of variables and a set \( C \) of clauses over \( X \) are specified, where each clause \( c\in C \) consists of three distinct variables, and each variable appears in exactly four clauses (the formula contains no negations).\\
\emph{Question:} Does there exist a truth assignment for \( X \) such that each clause contains at least one true and at least one false literal?
\end{minipage}}
\end{center}

Darmann and D{\"o}cker~\cite{darmann_docker} demonstrated the NP-Completeness of this problem. 

\begin{figure}[hbtp]
\centering
\includegraphics{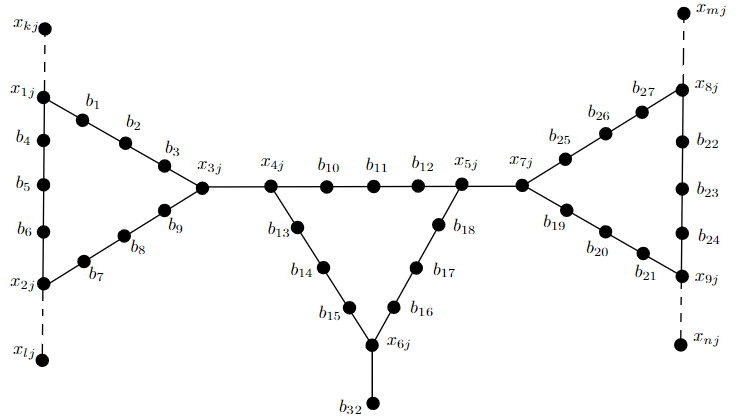}
\caption{Variable gadget \( G^X \).  
Vertex \( x_j \) together with incident edges are replaced by this gadget, and the edges \( e_1,e_2,e_3 \) and \( e_4 \) incident on \( x_j \) in \( G_\mathcal{F} \) become the dashed edges of the gadget in \( G \).}
\label{fig:4graceful variable gadget}
\end{figure}

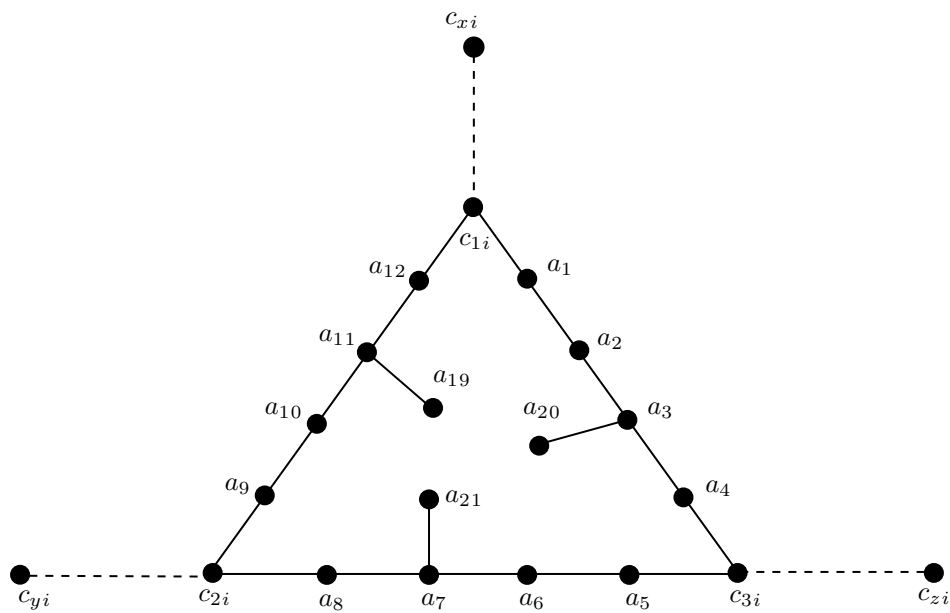
\begin{figure}[hbtp]
	\centering
\tikzset{every picture/.style={line width=0.75pt}} 

\begin{tikzpicture}[x=0.75pt,y=0.75pt,yscale=-1,xscale=1]

\draw   (339.88,65) -- (472.26,250.02) -- (207.5,250.02) -- cycle ;
\draw  [fill={rgb, 255:red, 0; green, 0; blue, 0 }  ,fill opacity=1 ] (467,249.5) .. controls (467,247.01) and (469.01,245) .. (471.5,245) .. controls (473.99,245) and (476,247.01) .. (476,249.5) .. controls (476,251.99) and (473.99,254) .. (471.5,254) .. controls (469.01,254) and (467,251.99) .. (467,249.5) -- cycle ;
\draw  [fill={rgb, 255:red, 0; green, 0; blue, 0 }  ,fill opacity=1 ] (413,250.5) .. controls (413,248.01) and (415.01,246) .. (417.5,246) .. controls (419.99,246) and (422,248.01) .. (422,250.5) .. controls (422,252.99) and (419.99,255) .. (417.5,255) .. controls (415.01,255) and (413,252.99) .. (413,250.5) -- cycle ;
\draw  [fill={rgb, 255:red, 0; green, 0; blue, 0 }  ,fill opacity=1 ] (362,250.5) .. controls (362,248.01) and (364.01,246) .. (366.5,246) .. controls (368.99,246) and (371,248.01) .. (371,250.5) .. controls (371,252.99) and (368.99,255) .. (366.5,255) .. controls (364.01,255) and (362,252.99) .. (362,250.5) -- cycle ;
\draw  [fill={rgb, 255:red, 0; green, 0; blue, 0 }  ,fill opacity=1 ] (313,250.5) .. controls (313,248.01) and (315.01,246) .. (317.5,246) .. controls (319.99,246) and (322,248.01) .. (322,250.5) .. controls (322,252.99) and (319.99,255) .. (317.5,255) .. controls (315.01,255) and (313,252.99) .. (313,250.5) -- cycle ;
\draw  [fill={rgb, 255:red, 0; green, 0; blue, 0 }  ,fill opacity=1 ] (262,250.5) .. controls (262,248.01) and (264.01,246) .. (266.5,246) .. controls (268.99,246) and (271,248.01) .. (271,250.5) .. controls (271,252.99) and (268.99,255) .. (266.5,255) .. controls (264.01,255) and (262,252.99) .. (262,250.5) -- cycle ;
\draw  [fill={rgb, 255:red, 0; green, 0; blue, 0 }  ,fill opacity=1 ] (205,249.5) .. controls (205,247.01) and (207.01,245) .. (209.5,245) .. controls (211.99,245) and (214,247.01) .. (214,249.5) .. controls (214,251.99) and (211.99,254) .. (209.5,254) .. controls (207.01,254) and (205,251.99) .. (205,249.5) -- cycle ;
\draw  [fill={rgb, 255:red, 0; green, 0; blue, 0 }  ,fill opacity=1 ] (231,210.5) .. controls (231,208.01) and (233.01,206) .. (235.5,206) .. controls (237.99,206) and (240,208.01) .. (240,210.5) .. controls (240,212.99) and (237.99,215) .. (235.5,215) .. controls (233.01,215) and (231,212.99) .. (231,210.5) -- cycle ;
\draw  [fill={rgb, 255:red, 0; green, 0; blue, 0 }  ,fill opacity=1 ] (257,174.5) .. controls (257,172.01) and (259.01,170) .. (261.5,170) .. controls (263.99,170) and (266,172.01) .. (266,174.5) .. controls (266,176.99) and (263.99,179) .. (261.5,179) .. controls (259.01,179) and (257,176.99) .. (257,174.5) -- cycle ;
\draw  [fill={rgb, 255:red, 0; green, 0; blue, 0 }  ,fill opacity=1 ] (282,138.5) .. controls (282,136.01) and (284.01,134) .. (286.5,134) .. controls (288.99,134) and (291,136.01) .. (291,138.5) .. controls (291,140.99) and (288.99,143) .. (286.5,143) .. controls (284.01,143) and (282,140.99) .. (282,138.5) -- cycle ;
\draw  [fill={rgb, 255:red, 0; green, 0; blue, 0 }  ,fill opacity=1 ] (308,102.5) .. controls (308,100.01) and (310.01,98) .. (312.5,98) .. controls (314.99,98) and (317,100.01) .. (317,102.5) .. controls (317,104.99) and (314.99,107) .. (312.5,107) .. controls (310.01,107) and (308,104.99) .. (308,102.5) -- cycle ;
\draw  [fill={rgb, 255:red, 0; green, 0; blue, 0 }  ,fill opacity=1 ] (335,65.5) .. controls (335,63.01) and (337.01,61) .. (339.5,61) .. controls (341.99,61) and (344,63.01) .. (344,65.5) .. controls (344,67.99) and (341.99,70) .. (339.5,70) .. controls (337.01,70) and (335,67.99) .. (335,65.5) -- cycle ;
\draw  [fill={rgb, 255:red, 0; green, 0; blue, 0 }  ,fill opacity=1 ] (362,101.5) .. controls (362,99.01) and (364.01,97) .. (366.5,97) .. controls (368.99,97) and (371,99.01) .. (371,101.5) .. controls (371,103.99) and (368.99,106) .. (366.5,106) .. controls (364.01,106) and (362,103.99) .. (362,101.5) -- cycle ;
\draw  [fill={rgb, 255:red, 0; green, 0; blue, 0 }  ,fill opacity=1 ] (388,137.5) .. controls (388,135.01) and (390.01,133) .. (392.5,133) .. controls (394.99,133) and (397,135.01) .. (397,137.5) .. controls (397,139.99) and (394.99,142) .. (392.5,142) .. controls (390.01,142) and (388,139.99) .. (388,137.5) -- cycle ;
\draw  [fill={rgb, 255:red, 0; green, 0; blue, 0 }  ,fill opacity=1 ] (412,172.5) .. controls (412,170.01) and (414.01,168) .. (416.5,168) .. controls (418.99,168) and (421,170.01) .. (421,172.5) .. controls (421,174.99) and (418.99,177) .. (416.5,177) .. controls (414.01,177) and (412,174.99) .. (412,172.5) -- cycle ;
\draw  [fill={rgb, 255:red, 0; green, 0; blue, 0 }  ,fill opacity=1 ] (440,211.5) .. controls (440,209.01) and (442.01,207) .. (444.5,207) .. controls (446.99,207) and (449,209.01) .. (449,211.5) .. controls (449,213.99) and (446.99,216) .. (444.5,216) .. controls (442.01,216) and (440,213.99) .. (440,211.5) -- cycle ;
\draw[dashed] (339.88,65)--++(0,-80) node[circle,fill,inner sep=0pt,minimum size=8pt]{};
\draw  [fill={rgb, 255:red, 0; green, 0; blue, 0 }  ,fill opacity=1 ] (565,249.5) .. controls (565,247.01) and (567.01,245) .. (569.5,245) .. controls (571.99,245) and (574,247.01) .. (574,249.5) .. controls (574,251.99) and (571.99,254) .. (569.5,254) .. controls (567.01,254) and (565,251.99) .. (565,249.5) -- cycle ;
\draw[dashed]    (472.88,249) -- (569.5,249) ;
\draw  [fill={rgb, 255:red, 0; green, 0; blue, 0 }  ,fill opacity=1 ] (117.87,250.43) .. controls (117.85,252.91) and (115.83,254.92) .. (113.34,254.91) .. controls (110.86,254.9) and (108.85,252.87) .. (108.87,250.39) .. controls (108.88,247.9) and (110.9,245.9) .. (113.39,245.91) .. controls (115.87,245.92) and (117.88,247.94) .. (117.87,250.43) -- cycle ;
\draw[dashed]    (209.98,251.35) -- (113.36,250.91) ;
\draw    (286.5,138.5) -- (321.5,169) ;
\draw    (371.5,185) -- (416.5,172.5) ;
\draw    (317.5,211) -- (317.5,250.5) ;
\draw  [fill={rgb, 255:red, 0; green, 0; blue, 0 }  ,fill opacity=1 ] (313,212.5) .. controls (313,210.01) and (315.01,208) .. (317.5,208) .. controls (319.99,208) and (322,210.01) .. (322,212.5) .. controls (322,214.99) and (319.99,217) .. (317.5,217) .. controls (315.01,217) and (313,214.99) .. (313,212.5) -- cycle ;
\draw  [fill={rgb, 255:red, 0; green, 0; blue, 0 }  ,fill opacity=1 ] (368,185.5) .. controls (368,183.01) and (370.01,181) .. (372.5,181) .. controls (374.99,181) and (377,183.01) .. (377,185.5) .. controls (377,187.99) and (374.99,190) .. (372.5,190) .. controls (370.01,190) and (368,187.99) .. (368,185.5) -- cycle ;
\draw  [fill={rgb, 255:red, 0; green, 0; blue, 0 }  ,fill opacity=1 ] (315,166.5) .. controls (315,164.01) and (317.01,162) .. (319.5,162) .. controls (321.99,162) and (324,164.01) .. (324,166.5) .. controls (324,168.99) and (321.99,171) .. (319.5,171) .. controls (317.01,171) and (315,168.99) .. (315,166.5) -- cycle ;

\draw (331,77.4) node [anchor=north west][inner sep=0.75pt]  [font=\footnotesize]  {$c_{1}{}_{i}$};
\draw (201,256.4) node [anchor=north west][inner sep=0.75pt]  [font=\footnotesize]  {$c_{2}{}_{i}$};
\draw (110.87,256.79) node [anchor=north west][inner sep=0.75pt]  [font=\footnotesize]  {$c_{y}{}_{i}$};
\draw (466,256.9) node [anchor=north west][inner sep=0.75pt]  [font=\footnotesize]  {$c_{3}{}_{i}$};
\draw (560,255.9) node [anchor=north west][inner sep=0.75pt]  [font=\footnotesize]  {$c_{z}{}_{i}$};
\draw (324,-34) node [anchor=north west][inner sep=0.75pt]  [font=\footnotesize]  {$c_{x}{}_{i}$};
\draw (375,90.4) node [anchor=north west][inner sep=0.75pt]  [font=\footnotesize]  {$a_{1}$};
\draw (400,127.4) node [anchor=north west][inner sep=0.75pt]  [font=\footnotesize]  {$a_{2}$};
\draw (425,162.4) node [anchor=north west][inner sep=0.75pt]  [font=\footnotesize]  {$a_{3}$};
\draw (454,201.4) node [anchor=north west][inner sep=0.75pt]  [font=\footnotesize]  {$a_{4}$};
\draw (414,258.4) node [anchor=north west][inner sep=0.75pt]  [font=\footnotesize]  {$a_{5}$};
\draw (361,258.4) node [anchor=north west][inner sep=0.75pt]  [font=\footnotesize]  {$a_{6}$};
\draw (312,258.4) node [anchor=north west][inner sep=0.75pt]  [font=\footnotesize]  {$a_{7}$};
\draw (261,258.9) node [anchor=north west][inner sep=0.75pt]  [font=\footnotesize]  {$a_{8}$};
\draw (214,200.4) node [anchor=north west][inner sep=0.75pt]  [font=\footnotesize]  {$a_{9}$};
\draw (234,162.4) node [anchor=north west][inner sep=0.75pt]  [font=\footnotesize]  {$a_{1}{}_{0}$};
\draw (261,125.4) node [anchor=north west][inner sep=0.75pt]  [font=\footnotesize]  {$a_{1}{}_{1}$};
\draw (286,91.4) node [anchor=north west][inner sep=0.75pt]  [font=\footnotesize]  {$a_{1}{}_{2}$};
\draw (318,146.4) node [anchor=north west][inner sep=0.75pt]  [font=\footnotesize]  {$a_{1}{}_{9}$};
\draw (363,162.4) node [anchor=north west][inner sep=0.75pt]  [font=\footnotesize]  {$a_{2}{}_{0}$};
\draw (324,206.4) node [anchor=north west][inner sep=0.75pt]  [font=\footnotesize]  {$a_{2}{}_{1}$};

\end{tikzpicture}
\caption{Clause gadget \( G^{C} \). 
Vertex \( c_i \) together with incident edges are replaced by this gadget, and the edges \( e_1,e_2 \) and \( e_3 \) incident on \( c_i \) in \( G_\mathcal{F} \) become the dashed edges of the gadget in \( G \).}
\label{fig:4graceful clause gadget}
\end{figure}

\begin{theorem}
\textsc{Graceful 4-colorablity} is NP-complete for the class of 2-degenerate graphs of maximum degree 3. 
\end{theorem}
\begin{proof}[Proof overview]
Given a boolean formula \( \mathcal{F}=(X,C) \) which is an instance of \textsc{Positive Not-All-Equal $3$-Sat E4}, we construct a graph \( G \) by taking the variable-clause incidence graph \( G_\mathcal{F} \) of the formula \( \mathcal{F} \), and then replacing each variable vertex \( x_j \) and incident edges by the variable gadget \( G^X \) (see Figure~\ref{fig:4graceful variable gadget}) and replacing each clause vertex \( c_i \) and incident edges by the clause gadget \( G^C \) (see Figure~\ref{fig:4graceful clause gadget}).

We observe that for every graceful 4-coloring \( f \) of \( G \),\\ 
(i)~\( f(c_{1i}), f(c_{2i}), f(c_{3i})\in \{1,4\} \) and not all equal for each \( i \), and\\
(ii)~\( f(x_{lj})=f(x_{kj})=f(x_{mj})=f(x_{nj})\in \{1,4\} \) for all \( j \).\\ 
Using these properties, we show that \( \mathcal{F} \) is a yes instance of \textsc{Positive Not-All-Equal $3$-Sat E4} whenever \( G \) admits a graceful 4-coloring \( f \). 
In the other direction, we show that if \( \mathcal{F} \) is a yes instance of \textsc{Positive Not-All-Equal $3$-Sat E4}, then a graceful 4-coloring of \( G \) can be obtained using the colorings schemes shown in Figure~\ref{fig:graceful4col of variable gadget} and Figure~\ref{fig:graceful4col of clause gadget}. 
Consider a truth assignment \( \mathcal{A} \) for the formula \( \mathcal{F} \) such that each clause contains a true variable and a false variable. 
If a variable \( x_j \) is true under \( \mathcal{A} \), then color the corresponding variable gadget \( G^X \) by the coloring \( h \) shown in Figure~\ref{fig:graceful4col of variable gadget}. 
If a variable \( x_j \) is false under \( \mathcal{A} \), then color the corresponding variable gadget \( G^X \) using the `opposite' colors of that shown in Figure~\ref{fig:graceful4col of variable gadget} (i.e., color each vertex \( v \) in the gadget using the color \( 5-h(v) \)). 
This can be extended to a graceful 4-coloring of \( G \) by coloring each clause gadget by one of the colorings in Figure~\ref{fig:graceful4col of clause gadget} or a suitable angular rotation of them. 
\end{proof}

\begin{figure}[hbtp]
\centering
\includegraphics[width=\textwidth]{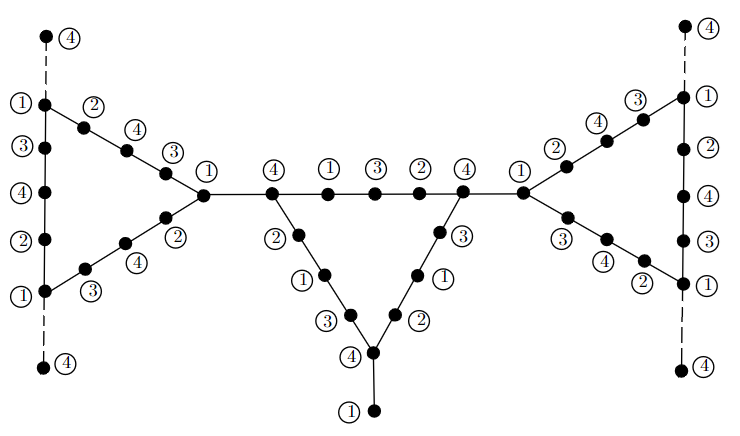}
\caption{A graceful 4-coloring \( h \) of \( G^{X} \).}
\label{fig:graceful4col of variable gadget}
\end{figure}

\begin{figure}[hbtp]
\centering
\includegraphics[width=\textwidth]{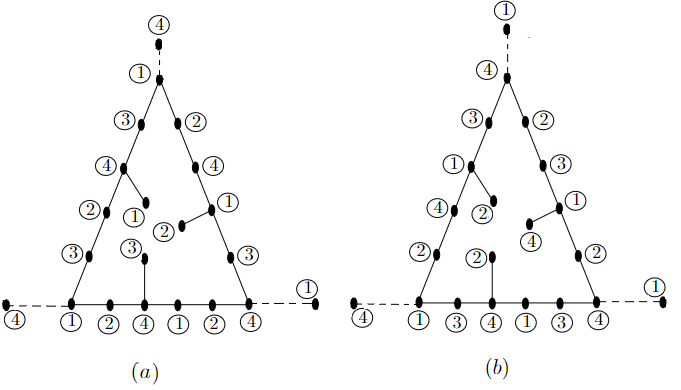}
\caption{Two graceful 4-colorings of \( G^C \)}
\label{fig:graceful4col of clause gadget}
\end{figure}

\FloatBarrier

\printbibliography

%
%
%
%
\end{document}